\documentclass[a4paper,10pt]{amsart}

\usepackage[utf8]{inputenc}
\usepackage[T1]{fontenc}
\usepackage{amsmath,amssymb,amsthm}
\usepackage{mathrsfs}
\usepackage{url}
\usepackage{siunitx}

\usepackage[pagebackref = true]{hyperref}
\hypersetup{
colorlinks = true,
pdfstartview = FitH,
pdfpagemode = UseNone,
pdfauthor    = {Antal Balog, Andras Biro, Giacomo Cherubini, Niko Laaksonen},
pdftitle     = {Bykovskii-type theorem for the Picard manifold},
pdfkeywords  = {prime geodesic theorem, L-functions, Kloosterman sums, zero-density estimates},
pdfcreator   = {Pdflatex},
linktoc = none,
bookmarks = false
}

\numberwithin{equation}{section}

\def\eps{\epsilon}

\def\mfrak{\mathfrak{m}}
\def\pfrak{\mathfrak{p}}

\def\ZZ{\mathbb{Z}}
\def\QQ{\mathbb{Q}}
\def\RR{\mathbb{R}}
\def\CC{\mathbb{C}}

\def\Acal{\mathcal{A}}
\def\Ecal{\mathcal{E}}
\def\Ical{\mathcal{I}}
\def\Mcal{\mathcal{M}}

\def\Lscr{\mathscr{L}}
\def\Cscr{\mathscr{C}}

\def\pslz{\mathrm{PSL}(2,\ZZ)}
\def\pslzi{\mathrm{PSL}(2,\ZZ[i])}

\theoremstyle{plain}
\newtheorem{theorem}{Theorem}[section]
\newtheorem{lemma}[theorem]{Lemma}
\newtheorem{corollary}[theorem]{Corollary}

\theoremstyle{definition}
\newtheorem{remark}{Remark}

\title{Bykovskii-type theorem for the Picard manifold}

\author{Antal Balog}
\author{Andr\'as Bir\'o}
\author{Giacomo Cherubini}
\author{Niko Laaksonen}

\address{
    Alfr\'ed R\'enyi Institute of Mathematics,
    POB 127, Budapest H-1364, Hungary;
    R\'enyi Int\'ezet Lend\"ulet Automorphic Research Group
}
\email{
    balog.antal@renyi.hu\\
    biroand@renyi.hu\\
    cherubini.giacomo@renyi.hu\\
    laaksonen.niko@renyi.hu
}

\thanks{This work was supported by a R\'enyi Int\'ezet Lend\"ulet Automorphic
    Research Group and by
    NKFIH (National Research, Development and Innovation Office)
grants K 109789, K 119528.}
\date{\today}

\begin{document}

\begin{abstract}
    We generalise a result of Bykovskii to the Gaussian integers
    and prove an asymptotic formula for the prime geodesic theorem
    in short intervals on the Picard manifold.
    Previous works show that
    individually the remainder is bounded by
    $O(X^{13/8+\epsilon})$ and $O(X^{3/2+\theta+\epsilon})$,
where $\theta$ is the subconvexity exponent for quadratic Dirichlet $L$-functions over $\mathbb{Q}(i)$.
    By combining arithmetic methods with estimates
    for a spectral exponential sum
    and a smooth explicit formula,
    we obtain an improvement for both of these exponents.
Moreover, by assuming two standard conjectures on $L$-functions,
we show that it is possible to reduce the exponent below the barrier $3/2$
and get $O(X^{34/23+\epsilon})$ conditionally.
    We also demonstrate a dependence of the remainder in the short interval
    estimate on the classical Gauss circle problem for shifted centres.
\end{abstract}
\maketitle


\section{Introduction}

It is well-known that the lengths of prime geodesics on a hyperbolic surface
behave similarly to prime numbers when counted according to size.
In 1997, Bykovskii~\cite{bykovskii_density_1997} proved that the analogy
holds also in short intervals on the modular surface, thus
resolving a conjecture of Iwaniec~\cite[\S4]{iwaniec_1984}.
More precisely, let
\[\pi_\Gamma(X) = \sum_{N(P)\leq X}\!\! 1,\]
where the sum runs over primitive hyperbolic conjugacy classes of
$\Gamma=\pslz$ of norm at most $X$.
Bykovskii~\cite[Theorem~1]{bykovskii_density_1997} showed that,
for every $\eps>0$,
\begin{equation}\label{eq:bykovskii}
    \pi_{\Gamma}(X+Y)-\pi_{\Gamma}(X)=\int_{X}^{X+Y}\frac{du}{\log
    u}+O(YX^{-\sigma(\nu)+\epsilon}),
\end{equation}
where $Y=X^{\nu}$, $1/2<\nu\leq 1$ and $\sigma(\nu)>0$.
Moreover, he observed
that~\eqref{eq:bykovskii} is in fact optimal in the sense that it is not
possible to reduce $\nu$ below 1/2.

In this paper we consider Bykovskii's problem on the three-dimensional analogue
of the modular surface---the Picard manifold $\Gamma\backslash\mathbb{H}^{3}$,
where $\Gamma=\pslzi$ is the Picard group and $\mathbb{H}^{3}$ is the
upper half-space. In this case $\pi_{\Gamma}$ counts not only hyperbolic,
but also loxodromic (i.e.~with non-real trace) conjugacy classes of
$\Gamma$. As in the theory of prime numbers, it is more convenient to consider
the related Chebyshev-type weighted counting function
\[
    \Psi_\Gamma(X) = \sum_{N(P)\leq X} \Lambda_\Gamma(N(P)),
\]
where the sum is now over all hyperbolic and loxodromic conjugacy
classes of $\Gamma$,
and we define $\Lambda_\Gamma(N(P))=\log N(P_0)$ if $\{P_0\}$
is the primitive conjugacy class associated to $\{P\}$, and
$\Lambda_\Gamma(N(P))=0$
otherwise. Notice that $\log N(P_{0})$ is the length of the closed geodesic
corresponding to $\{P_{0}\}$.
We refer to \cite[\S5.7]{elstrodt_groups_1998} and \cite[\S2]{balkanova_prime_2018-1}
for more detailed terminology and definitions.

A seminal result of Sarnak \cite[Theorem 5.1]{sarnak_arithmetic_1983}
gives an asymptotic with error term for $\Psi_\Gamma(X)$ (and in fact for any
cofinite $\Gamma$), namely
\begin{equation}\label{eq:sarnak}
    \Psi_\Gamma(X) = \tfrac{1}{2}X^2 +O(X^{5/3+\eps}),
\end{equation}
for every $\eps>0$.
There have been several improvements of~\eqref{eq:sarnak} for the Picard group.
Koyama~\cite{koyama_prime_2006} proved, conditionally on a mean Lindel\"of
hypothesis for certain automorphic $L$-functions (see
\eqref{intro:meanlindelof}),
that the error can be improved to $O(X^{11/7+\epsilon})$.
This was later strengthened by Balkanova and
Frolenkov~\cite{balkanova_prime_2018} to
$O(X^{3/2+\theta+\epsilon})$, where $\theta$ is the subconvexity exponent of
quadratic Dirichlet $L$-functions
over $\QQ(i)$ (see~\eqref{intro:subconv}).
In the recent work~\cite{balkanova_prime_2018-1}, Sarnak's exponent was
unconditionally improved to $13/8+\epsilon$.

Our main result provides an asymptotic formula with a power saving for
the localised problem in which one considers
the difference $\Psi_\Gamma(X+Y)-\Psi_\Gamma(X)$.

\begin{theorem}\label{intro:theorem1}
    Fix $\nu\in(\frac{1}{3},1]$ and let $Y=X^{\nu}$, $X\gg 1$. Then
    \begin{equation}\label{intro:thm:eq}
        \Psi_\Gamma(X+Y) - \Psi_\Gamma(X) = XY + \tfrac{1}{2}Y^2 +
        O((XY)X^{-\beta(\nu)+\eps}),
    \end{equation}
    where $\beta(\nu)>0$ is defined in~\eqref{eq:beta}.
    Moreover, there exists $\eta\in(\frac{1}{4},\frac{1}{3})$
    such that, for fixed $\nu\in(\eta,1]$, and for $Y=X^\nu$,
    we may replace the remainder by
    \begin{equation}\label{intro:thm:eq2}
        O((XY)X^{-\alpha(\nu,\eta)+\eps}),
    \end{equation}
    where $\alpha(\nu,\eta)>0$ is defined in \eqref{0809:eq002}.
\end{theorem}

The proof of Theorem~\ref{intro:theorem1} follows the general outline of Bykovskii's method
with due adaptations to the setting of $\QQ(i)$.
We also modify slightly the argument of \cite[Lemma 5]{bykovskii_density_1997}
and give a version of the proof independent of
Kloosterman sums
(see the second bound in Lemma \ref{1606:lemma001}).

\begin{remark}
    The number $\eta$ is related to the Gauss circle problem as explained
    in~\eqref{intro:eq012}.
	Notice that Theorem~\ref{intro:theorem1} allows us to
    consider very short intervals with $\nu<1/2$.
    This is in contrast to the situation in two dimensions
    where shifts of the size $Y=X^{1/2+\epsilon}$ are optimal.
    Also, the remainder~\eqref{intro:thm:eq2} allows us
    to consider shorter intervals than those in~\eqref{intro:thm:eq}.
    However, concerning our applications of Theorem~\ref{intro:theorem1},
    the first estimate~\eqref{intro:thm:eq} is always stronger
	in the critical range of $\nu$.
\end{remark}

Let $E_{\Gamma}(X)= \Psi_{\Gamma}(X)-\frac{1}{2}X^{2}$ denote the remainder
in~\eqref{eq:sarnak}.
We can then combine Theorem~\ref{intro:theorem1} with estimates
for a certain spectral exponential sum
to obtain pointwise bounds for $E_{\Gamma}(X)$ by following
the ideas of Soundararajan and Young~\cite{soundararajan_prime_2013}.
\begin{corollary}\label{cor:uncond}
    For $X\gg 1$, we have
    \[E_{\Gamma}(X)\ll X^{13/8 - \beta/2+\epsilon},\]
    where $\beta = (177-\sqrt{\num{31 049}})/16$ and
    $13/8-\beta/2\approx1.60023$.
\end{corollary}

The number $\eta$ that appears in Theorem~\ref{intro:theorem1}
comes from the power saving for the remainder in the Gauss circle problem.
More precisely, we require a bound for the shifted circle problem.
Let $b\in\RR^2$,
and let $B(b,\sqrt{M})$ be the closed ball
centred at $b$ of radius $\sqrt{M}$.
By a simple geometric argument, one can see that
\begin{equation}\label{intro:eq012}
    |\ZZ^{2}\cap B(b,\sqrt{M})|
    =
    \pi M + O(M^{\eta+\eps}),
\end{equation}
for some $\eta\leq 1/2$. This is of course equivalent to counting points of
$b+\ZZ[i]$ inside $B(0,\sqrt{M})$.
It is expected that $\eta=1/4$, which, if true, would
be optimal.
The current best result is due to Huxley
\cite[Theorem 5]{huxley_exponential_2003}, who showed that $\eta=131/416$
is allowed uniformly in $b$.
Therefore, Theorem~\ref{intro:theorem1} holds unconditionally
with this value of $\eta$.
The connection of Theorem~\ref{intro:theorem1} to the Gauss circle problem
arises fairly naturally in our proof since we have to estimate sums over the
Gaussian integers.
However, we observe for the first time
a direct influence of the circle problem on the remainder in the prime geodesic theorem
(cf.~\cite{balkanova_prime_2018-1,balkanova_prime_2018,koyama_prime_2006},
where only the trivial bound in \eqref{intro:eq012} is needed).

Another important ingredient in our proof is a zero-density
theorem for the family of Dirichlet $L$-functions $L(s,\chi_D)$,
where $\chi_D$ is the Kronecker symbol over $\ZZ[i]$
(see \S\ref{S2.1}).
On the other hand, it is possible to bypass zero-density estimates and
simply keep track of the subconvexity exponent $\theta\in[0, 1/4]$, which
satisfies
\begin{equation}\label{intro:subconv}
    L(\tfrac{1}{2}+it,\chi_D) \ll (1+|t|)^A N(D)^{\theta+\eps},
\end{equation}
for all primitive quadratic characters $\chi_{D}$ over $\QQ(i)$
and for some $A>0$. The convexity bound corresponds to $\theta=1/4$, while
the Lindel\"of hypothesis would yield $\theta=0$.
Then, together with the conjectural bound for the Gauss
circle problem ($\eta=1/4$ in~\eqref{intro:eq012}),
we obtain the following variant of Theorem~\ref{intro:theorem1}.

\begin{theorem}\label{theorem:conditional}
    Let $\theta$ be the subconvexity exponent of quadratic Dirichlet
    $L$-functions over $\QQ(i)$ as defined in~\eqref{intro:subconv}.
    Then, for every $\nu\in(\frac{1}{3}, 1]$ with $Y=X^{\nu}$ and $X\gg 1$,
    we have
    \begin{equation}\label{eq:shortintervalsub}
        \Psi_\Gamma(X+Y) - \Psi_\Gamma(X) = XY + \tfrac{1}{2}Y^2
        +
        O(X^{(4\theta+6)/5+\epsilon}Y^{2/5}).
    \end{equation}
    If we assume the Lindel\"of hypothesis (i.e.~$\theta=0$) and furthermore
    \eqref{intro:eq012} with $\eta=1/4$,
    then, for every $\nu\in(\frac{1}{4},1]$, $X\gg 1$, and $Y=X^\nu$, we have
    \begin{equation}\label{eq:shortintervalgauss}
        \Psi_\Gamma(X+Y) - \Psi_\Gamma(X) = XY + \tfrac{1}{2}Y^2
        +
        O(X^{11/10+\eps}Y^{3/5}).
    \end{equation}
\end{theorem}

In Theorem~\ref{theorem:conditional} the first
equation~\eqref{eq:shortintervalsub} again follows from
treating Kloosterman sums, while in~\eqref{eq:shortintervalgauss} we replace
them with the bound for the Gauss circle problem.
It is interesting to notice that if we assume the Lindel\"of hypothesis also
in~\eqref{eq:shortintervalsub}, then this is stronger
than~\eqref{eq:shortintervalgauss} as long as $Y>X^{1/2}$.
Finally, we can of course use Theorem~\ref{theorem:conditional} to deduce pointwise
bounds.

\begin{corollary}\label{cor:11/7}
    Let $\theta$ denote the subconvexity exponent for $L(s,\chi_{D})$. Then,
    for $X\gg 1$, we have
    \begin{equation}\label{2808:eq003}
        E_\Gamma(X) \ll X^{3/2+4\theta/7+\epsilon}.
    \end{equation}
	By \cite[Theorem 1.1]{nelson}, we can take $\theta=1/6$ and obtain the exponent $67/42$.
\end{corollary}

\begin{remark}
The estimate~\eqref{2808:eq003} improves upon the bound
in~\cite[Theorem~1.2]{balkanova_prime_2018} by reducing the coefficient in
front of $\theta$ from one to $4/7$.
In a recent breakthrough, Nelson~\cite{nelson} generalised
the Conrey--Iwaniec bound to Dirichlet $L$-functions over number fields,
which allows us to take $\theta=1/6$ in Corollary~\ref{cor:11/7}
and leads to a stronger estimate than in Corollary~\ref{cor:uncond}
(since $67/42\approx 1.59524$).
Nevertheless, the result in Corollary~\ref{cor:uncond} is of independent
interest as the method of proof is different.
\end{remark}

Conditionally, we can further improve \eqref{2808:eq003} if we assume
the same mean Lindel\"of hypothesis as Koyama~\cite{koyama_prime_2006}.
More precisely,
let $\lambda_{j}=1+r_{j}^{2}$ denote the eigenvalues of the Laplace--Beltrami
operator on $\Gamma\backslash\mathbb{H}^{3}$.
Also, let $u_{j}$ be the Maass cusp form corresponding to $r_{j}$.
Then, we assume that there exists $A>0$ such that,
for all $w\in\CC$ with $\Re(w)=1/2$,
we have the estimate
\begin{equation}\label{intro:meanlindelof}
\sum_{r_j\leq T} \frac{r_j}{\sinh(\pi r_j)}|L(w,u_j\otimes u_j)| \ll |w|^A T^{3+\eps},
\end{equation}
where $L(s,u_j\otimes u_j)$ is the Rankin--Selberg $L$-function associated
to $u_{j}$.
The following corollary shows that if \eqref{intro:meanlindelof} holds,
then we can reduce the exponent for $E_\Gamma(X)$ below $3/2$ as soon as
$\theta<1/24$.

\begin{corollary}\label{cor:34/23}
Let $\theta$ be as in \eqref{intro:subconv},
and assume \eqref{intro:meanlindelof}. Then, for $X\gg 1$, we have
\[
E_\Gamma(X) \ll X^{3/2+(24\theta-1)/46+\eps}.
\]
\end{corollary}

In particular, for $\theta=0$ we obtain $E_{\Gamma}(X)\ll X^{34/23+\epsilon}$.
It is unclear what should be the correct order of magnitude
of $E_\Gamma(X)$.
Corollary~\ref{cor:34/23} sheds some light on this
by showing that the exponent~$3/2$, which seemed to be a barrier
in \cite[Remarks~1.5~and~3.1]{balkanova_prime_2018-1}
and \cite[Theorem 1.1]{balkanova_prime_2018},
can be reduced under the assumption of two fairly standard
conjectures on $L$-functions.

\section{Background and Auxiliary Lemmas}

Unlike many other papers on the prime geodesic theorem
(see
e.g~\cite{balkanova_prime_2018-1,balkanova_prime_2018,koyama_prime_2006}),
we do not directly use the spectral theory of automorphic forms
and instead connect $\Psi_\Gamma(X)$ to certain $\mathrm{GL}_{1}$ $L$-functions.
This allows us to exploit the arithmetic structure of the
problem, which was also crucial for the proofs in \cite{bykovskii_density_1997}
and \cite{soundararajan_prime_2013}.
In sections~\S\ref{S2.1} and \S\ref{S2.2}
we introduce the $L$-functions we need,
i.e.~the Dedekind zeta function, Dirichlet $L$-functions
attached to Kronecker symbols,
and Zagier's $L$-function $\Lscr(s,\delta)$.
In \S\ref{S2.3} we approximate $\Lscr(s,\delta)$ and prove a lemma
on the quality of the approximation, see Lemma \ref{lemma:R_V},
which will be used in section~\ref{S3}.

\subsection{Dirichlet characters and $L$-functions}\label{S2.1}
Recall that the ring of integers of $\QQ(i)$ is $\ZZ[i]$ and
that the class number of $\QQ(i)$ is one, i.e.~every ideal is principal.
The Dedekind zeta function of $\QQ(i)$ is given, for $\Re(s)>1$, by
\begin{equation*}
    \zeta_{\QQ(i)}(s)
    =
    \sum_{a\neq 0} \frac{1}{N(a)^s}.
\end{equation*}
Note that typically in the literature this sum is
taken over non-zero ideals $(a)\subseteq\ZZ[i]$.
However, for the sake of brevity, we abuse notation and
denote ideals by their generators so that our sums are over elements of
$\ZZ[i]$ unless stated otherwise. This is slightly imprecise since for each
ideal there are four generators.
In order to recover the conventional definition
one could attach a factor $1/4$ to sums over elements (as was done in
e.g.~\cite[(3.22),~(3.23)]{balkanova_prime_2018}). Alternatively, one can
specify a choice of a generator for each ideal and sum over subsets
of $\ZZ[i]$ (see e.g.~\cite[p.~394]{szmidt_selberg_1983}, where the sums
are taken over the first quadrant).
We refrain from taking either approach
believing that the reader will still be able to follow the rest of the
paper without confusion.

Let $\mfrak$ be a non-zero ideal of $\ZZ[i]$. A Dirichlet character modulo $\mfrak$
is a group homomorphism
\[
    \chi_\mfrak : \mathrm{Cl}^{\mfrak} \longrightarrow S^1,
\]
where $\mathrm{Cl}^{\mfrak}$ is the narrow ray class group of modulus $\mfrak$
defined (for $\QQ(i)$) as the quotient
\[\mathrm{Cl}^{\mfrak} =
    \left.
        \raisebox{6pt}{
            \rule{0pt}{20pt}
            \ensuremath{
                \left\{\text{\parbox{115pt}{\centering fractional ideals $\mathcal{I}$ of $\QQ(i)$\\
        coprime to $\mfrak$}}\right\}}}
        \middle/
        \raisebox{-6pt}{\ensuremath{
                \left\{\text{\parbox{130pt}{\centering principal ideals
                            $\Acal\subseteq\ZZ[i]$ s.t.\\
                            $\Acal=(a)$ with $a\equiv 1\bmod\mfrak$
        }}\right\}.}}
    \right.
\]
The narrow ray class group plays the role
of $(\ZZ/m\ZZ)^\times$ when the base field is $\QQ$,
and reduces to it by taking $\mfrak=m\ZZ$ in the definition.

We are interested in quadratic characters
associated to the Kronecker symbol
\[
    \chi_D(n) = \left(\frac{n}{D}\right),
\]
where $n$ and $D$ are non-zero elements in $\ZZ[i]$
(see e.g.~\cite[\S14.2]{ireland_classical_1990}).
The function $\chi_D$ is a Dirichlet character of modulus $\mfrak=(D)$.
In analogy with the rational case,
if we sum over non-zero $n\in\ZZ[i]$, we obtain
a Dirichlet $L$-function
\[
    L(s,\chi_D) = \sum_{n\neq 0} \frac{\chi_D(n)}{N(n)^s}, \qquad \Re(s)>1.
\]
If $D$ is square-free then $\chi_D$ is primitive
and $L(s,\chi_D)$
extends to an entire function with a functional equation that relates the
values at $s$ and $1-s$.
The generalised Riemann hypothesis predicts that the non-trivial zeros
would lie on the critical line $\Re(s)=1/2$.

\subsection{Zagier's $L$-function}\label{S2.2}
In a paper from 1977, Zagier \cite{zagier_modular_1977} studied a certain
$L$-function associated to binary quadratic forms and related to
quadratic characters over $\ZZ$. This $L$-function appears in the study
of the prime geodesic theorem over the rationals
(see \cite{balkanova_bounds_nodate,balkanova_sums_nodate,bykovskii_density_1997,soundararajan_prime_2013}),
and its generalisation to the Gaussian integers is relevant
in the prime geodesic theorem for $\pslzi$ (see \cite{balkanova_prime_2018}).
In this section we introduce such a generalization and state
an asymptotic result for the average of its coefficients in Lemma \ref{1606:lemma001}.

Let $s\in\CC$ with $\Re(s)>1$
and suppose $\delta=n^2-4$ for some non-zero Gaussian integer $n$.
Then $\delta$ is a discriminant of a binary quadratic form over $\ZZ[i]$ and we 
consider the associated $L$-function (the first appearance of this function is perhaps
in Szmidt's paper \cite[\S3.5]{szmidt_selberg_1983})
\begin{equation}\label{1206:eq001}
    \Lscr(s,\delta)
    =
    \frac{\zeta_{\QQ(i)}(2s)}{\zeta_{\QQ(i)}(s)}
    \sum_{q\neq 0} \frac{\rho_q(\delta)}{N(q)^{s}}
    =
    \sum_{q\neq 0} \frac{\lambda_q(\delta)}{N(q)^s},
\end{equation}
where the sums are over the Gaussian integers, and the coefficients are given by
\begin{equation}\label{def:rholambda}
    \begin{gathered}
        \rho_q(\delta) = \#\{ \, x\pmod{2q}: x^2\equiv \delta\pmod{4q} \, \},
        \\
        \lambda_q(\delta) = \!\!\! \sum_{q_1^2q_2q_3=q} \mu(q_2)\rho_{q_3}(\delta),\rule{0pt}{13pt}
    \end{gathered}
\end{equation}
with $\mu$ denoting the M\"obius function over $\ZZ[i]$.
The series in \eqref{1206:eq001} are absolutely convergent for $\Re(s)>1$,
and the function $\Lscr(s,\delta)$ extends to a meromorphic function on $\CC$
with at most a pole at $s=1$. In fact,
up to multiplication by a Dirichlet polynomial,
$\Lscr(s,\delta)$ is the $L$-function associated with a quadratic Dirichlet character
of $\ZZ[i]$.
For non-zero $D,l\in\ZZ[i]$, define
\[
    T^{(D)}_l(s) =
    \sum_{d|l} \frac{\chi_D(d)\mu(d)}{N(d)^s}\sigma_{1/2-s}\left(\frac{l}{d}\right).
\]
Here $\chi_D$ is the Kronecker symbol over $\ZZ[i]$
and $\sigma_\xi$ is the divisor function given by
\[
    \sigma_{\xi}(n) = \sum_{d|n} N(d)^{\xi}.
\]

\begin{lemma}\label{lemma:szmidt}
    Let $\delta$ be as above
    and write $\delta\sim Dl^2$,
    where $D$ {is} a generator of the
	 discriminant of the field extension $\QQ(i)(\sqrt{\delta})$.
	 Then
    \[
        \Lscr(s,\delta)
        =
        T^{(D)}_l(s) \, L(s,\chi_D).
    \]
\end{lemma}

Note that here and in the rest of the paper
we write $a\sim b$ to indicate that
$a$ and $b$ are associates, i.e.~they are equal up to
multiplication by a unit in $\ZZ[i]$.

\begin{proof}
    See \cite[Proposition 6]{szmidt_selberg_1983}.
\end{proof}

We can evaluate partial sums of $\rho_q(\delta)$ and $\lambda_q(\delta)$
in an asymptotic form with an error term. The size of the error will depend
on available bounds for the remainder
in the counting of lattice points in shifted circles,
that is, on the exponent $\eta$ in \eqref{intro:eq012}.
For the proof we will also need Kloosterman sums, which are defined over
$\ZZ[i]$ as
\[
    S(m, n, c) =  \sum_{a\in (\ZZ[i]/(c))^\times} e\big(\langle
    m,a/c\rangle\big) e\big(\langle n,a^{-1}/c\rangle\big),
\]
where $m, n,c\in\mathbb{Z}[i]$, $c\neq 0$; $a^{-1}$ denotes the inverse of
$a$ modulo the ideal $(c)$; and $\langle x,y\rangle$ denotes the
standard inner product on $\RR^2\cong \CC$. The Kloosterman sums also satisfy
Weil's bound~\cite[(3.5)]{motohashi_1997}
\begin{equation}\label{eq:weil}
    S(m,n,c) \ll |(m,n,c)| d(c) N(c)^{1/2},
\end{equation}
where $d(c)$ is the number of divisors of $c$.

\begin{lemma}\label{1606:lemma001}
    Let $q\in\ZZ[i]$, $q\neq 0$ and fix $\eps>0$. Then, for $Z\geq 1$,
    we have
    \begin{equation}\label{1507:eq001}
        \begin{split}
            \sum_{0<N(n)\leq Z} \lambda_q(n^2-4)
            =
            \pi Z & \!\! \sum_{q_1^2q_2=q} \frac{\mu(q_2)}{N(q_2)}\\
                  &+
            O\bigl(\min\{Z^{1/3}N(q)^{1/3+\epsilon},
            Z^{\eta+\epsilon}N(q)^{1-\eta+\epsilon}\}\bigr),
        \end{split}
    \end{equation}
    where $\eta$ is as in \eqref{intro:eq012} and the implied constant does not
    depend on $q$.
    Unconditionally we can take $\eta=131/416$.
\end{lemma}

\begin{proof}
    We begin by proving the first bound in the minimum.
    The result is immediate if $N(q)\geq Z^{2}$, since
    $\lambda_{q}(n^{2}-4)\ll N(q)^{\epsilon}$ for every $n$.
    Assume therefore that $N(q)<Z^{2}$.
    In view of~\eqref{def:rholambda}, we first work with the sum
    \begin{equation}\label{eq:rdef}
        R(Z) = \sum_{N(n)\leq Z}\rho_{q}(n^{2}-4).
    \end{equation}
    We follow the general strategy of the proof of the classical $O(r^{2/3})$
    bound for the Gauss circle problem. Let $\ast$ denote the usual convolution
    on $\RR^{2}$ and define
	\[
    f(x) = f_{\Delta, Z}(x) =
	\frac{1}{\pi\Delta^{2}}(\mathbf{1}_{[0,\sqrt{Z}]}
	\ast\mathbf{1}_{[0,\Delta]})(|x|),
	\]
	for some
    $1/\sqrt{Z}<\Delta<\sqrt{Z}$.
    Consider the smoothed version of~\eqref{eq:rdef} given by
    \[S(Z, \Delta) = \sum_{n\in\ZZ[i]}\rho_{q}(n^{2}-4)f(n).\]
    Notice that
    \begin{equation}\label{eq:rsapprox}
        \lvert R(Z)-S(Z,\Delta)\rvert\ll
        |q|^{\epsilon}\sum_{n}\mathbf{1}_{[\sqrt{Z}-\Delta,\sqrt{Z}+\Delta]}(|n|)\ll
        |q|^{\epsilon}\sqrt{Z}\Delta,
    \end{equation}
    since $\rho_{q}(\delta)\ll |q|^{\epsilon}$.
    Splitting the sum in $S$ into residue classes and applying two dimensional
    Poisson summation gives
    \begin{equation}\label{eq:szmodb}
        S(Z, \Delta) =\frac{1}{N(q)} \sum_{b\,(q)}\rho_{q}(b^{2}-4)\sum_{k}
        e\Bigl(\Bigl\langle k,
        \frac{b\bar{q}}{N(q)}\Bigr\rangle\Bigr)\widehat{f}
        \left(\frac{kq}{N(q)}\right).
    \end{equation}
    To treat the sum over $b$, consider the definition of
    $\rho_q(\delta)$ in \eqref{def:rholambda}, and observe that there is
    a one-to-one correspondence between solutions $x$~mod~$(2q)$ of
    $x^2 \equiv n^2-4$~mod~$(4q)$ and solutions $y$~mod~$(q)$ of
    $y^2+yn+1\equiv 0$~mod~$(q)$ (to see this write $x=2y+n$).
    Note in particular that any such $y$ must be coprime to~$q$.
    Therefore we have
    \[
        \rho_q(n^2-4) = \#\{y~(\mathrm{mod}~q):~y^2+yn+1\equiv 0~(\mathrm{mod}~q)\}.
    \]
    With the above observation, we can identify Kloosterman sums
    in~\eqref{eq:szmodb} and get
    \begin{equation}\label{eq:szklooster}
        \sum_{b\,(q)}\rho_{q}(b^{2}-4)
        e\Bigl(\Bigl\langle k,
                \frac{b\bar{q}}{N(q)}\Bigr\rangle\Bigr)=
                \sum_{\substack{y\,(q)\\ (y, q)=1}}
        e\Bigl(\Bigl\langle k,
        \frac{-y-y^{-1}}{q}\Bigr\rangle\Bigr)=
        S(k, k, q).
    \end{equation}
    For $k=0$ we have $S(0, 0, q)=\varphi(q)$ with $\varphi$ being the Euler
    totient function on $\ZZ[i]$. Therefore~\eqref{eq:szmodb} becomes
    \begin{equation}\label{eq:szmain}
        S(Z, \Delta) =\frac{\varphi(q)}{N(q)}\pi Z+
        \frac{1}{N(q)}\sum_{k\neq 0 }S(k, k,
                     q)\widehat{f}\left(\frac{kq}{N(q)}\right),
    \end{equation}
    where we have used the fact that $\widehat{f}(0)=\pi Z$.
    For non-zero $k$ we bound $\widehat{f}$ in absolute value.
    We use~\cite[Theorem~IV.3.3]{steinweiss_1971} together
    with~\cite[5.52~(1)]{gradshteyn_2007}
    to see that
    \[\widehat{f}(x)=\frac{\sqrt{Z}}{\pi\Delta|x|^{2}}J_{1}(2\pi\sqrt{Z}|x|)
    J_{1}(2\pi\Delta|x|).\]
    Then, by~\cite[8.440~and~8.451~(1)]{gradshteyn_2007} it follows that $J_{1}(u)\ll
    \min(u, u^{-1/2})$ for $u>0$.
    Thus (since $\Delta<\sqrt{Z}$)
    \begin{equation*}
        \widehat{f}(x) \ll\begin{cases}
            Z, &\text{if $\sqrt{Z}|x|<1$,}\\
            Z^{1/4}|x|^{-3/2}, &\text{if $Z^{-1/2}\leq |x|<\Delta^{-1}$,}\\
            Z^{1/4}|x|^{-3}\Delta^{-3/2}, &\text{if $\Delta|x|>1$.}
        \end{cases}
    \end{equation*}
    Applying these estimates together with the Weil bound~\eqref{eq:weil} gives
    \begin{equation}\label{eq:szbound}
        \sum_{k\neq 0}S(k, k, q)\widehat{f}\left(\frac{kq}{N(q)}\right)
        \ll N(q)^{3/2+\epsilon}\bigl(1 +
        Z^{1/4}\Delta^{-1/2}\bigr),
    \end{equation}
    where we have used the facts that $d(q)\ll |q|^{\epsilon}$ and
    the gcd is one on average.
    Inserting~\eqref{eq:szmain} and~\eqref{eq:szbound} into~\eqref{eq:rsapprox}
    shows that
    \[R(Z) = \frac{\varphi(q)}{N(q)}\pi Z + O\bigl(|q|^{\epsilon}\sqrt{Z}\Delta +
    N(q)^{1/2+\epsilon}(1 + Z^{1/4}\Delta^{-1/2})\bigr).\]
    Recalling that $\sqrt{Z}/\Delta>1$ and then balancing with
    $\Delta=Z^{-1/6}N(q)^{1/3}$ finally gives
    \begin{equation}\label{eq:firstrbound}
        R(Z) = \frac{\varphi(q)}{N(q)}\pi Z + O(Z^{1/3}N(q)^{1/3+\epsilon}).
    \end{equation}

    To get back to the statement of the lemma, we
    combine~\eqref{eq:firstrbound} with~\eqref{def:rholambda} and write
    \[
        \sum_{N(n)\leq Z} \lambda_q(n^2-4)
        =
        \pi Z \sum_{q_1^2q_2q_3=q} \mu(q_2) \frac{\varphi(q_3)}{N(q_3)}
        + O(Z^{1/3+\eps}N(q)^{1/3+\eps}).
    \]
    The sum on the right-hand side is handled by using the identity
    \[
        \frac{\varphi(q_3)}{N(q_3)} = \sum_{d|q_3} \frac{\mu(d)}{N(d)},
    \]
    which gives (if we let $q_4=q_3/d$)
    \[
        \sum_{q_1^2q_2q_4d=q} \frac{\mu(q_2)\mu(d)}{N(d)}
        =
        \sum_{q_1^2d|q} \frac{\mu(d)}{N(d)} \sum_{q_2|q/q_1^2d} \mu(q_2)
        =
        \sum_{q_1^2d=q} \frac{\mu(d)}{N(d)}.
    \]
    Up to renaming $d$ as $q_2$, the last expression is what appears in~\eqref{1507:eq001}.

    We now prove the second bound in the minimum in~\eqref{1507:eq001}.
    In this case we can assume $Z>N(q)$.
    Instead of smoothing~\eqref{eq:rdef}, we separate the sum into
    residue classes so that
    \begin{equation}\label{eq:rmodq}
        R(Z)
        =
        \sum_{b\,(q)} \rho_q(b^2-4) \sum_{\substack{N(n)\leq Z\\ n\equiv
        b\,(q)}} \!\! 1.
    \end{equation}
    The idea is then to estimate each of the circles separately thus completely
    avoiding having to treat Kloosterman sums.

    A given residue class $b$ modulo $q$ corresponds to the set $b+q\ZZ[i]$.
    Up to rescaling by $|q|=\sqrt{N(q)}$, this can be identified with a copy
    of $\ZZ[i]$ rotated by $\arg(q)$ and translated by $b/|q|$.
    By \eqref{intro:eq012}, we obtain
    \begin{equation}\label{1507:eq002}
        \sum_{\substack{N(n)\leq Z\\ n\equiv b\,(q)}} \!\! 1
        \,=\,
        \pi M^2 + O(M^{2\eta+\eps}),
        \qquad
        M=\sqrt{\frac{\mathstrut Z}{\smash{N(q)}}},
    \end{equation}
    where the implied constant is independent of $b$ and $q$.
    For the sum over $b$ we use the first equation in~\eqref{eq:szklooster}
    with $k=0$ to deduce that
    \begin{equation}\label{eq:rhosum}
        \sum_{b\,(q)}\rho_q(b^2-4) = \varphi(q).
    \end{equation}
    Applying~\eqref{1507:eq002} and~\eqref{eq:rhosum} in~\eqref{eq:rmodq}
    yields
    \begin{equation}\label{1507:eq003}
        R(Z) = \frac{\varphi(q)}{N(q)}\pi Z +
        O(Z^{\eta+\eps}N(q)^{1-\eta+\eps}).
    \end{equation}
    Finally,~\eqref{1507:eq003} can be related to the sum
    in~\eqref{1507:eq001} just as in the previous case.
    The final assertion with the unconditional result
    follows from \cite[Theorem~5]{huxley_exponential_2003},
    where it is proved that $\eta=131/416$ is admissible.
\end{proof}

\subsection{Zero-density estimates for Dirichlet $L$-functions over $\QQ(i)$}\label{S2.3}
In our proof we need to evaluate the $L$-function $\Lscr(s,\delta)$ at $s=1$.
Since the Dirichlet series is only conditionally convergent at this point,
we approximate $\Lscr(1,\delta)$ by an absolutely convergent series with
an exponential weight and give bounds on the error arising in the
process.
We do this by applying
a zero-density theorem for Dirichlet $L$-functions over number fields
due to Huxley \cite{huxley_large_1971}.

Let $V>0$, and consider the integral
\[
    \frac{1}{2\pi i} \int_{(1+\eps)} \Lscr(s,\delta)\Gamma(s-1)V^{s-1}\,ds.
\]
Let $1/2\leq\sigma<1$ and move the line of integration to $\Re(s)=\sigma$.
We pass
a pole at $s=1$ and obtain
\[
    \Lscr(1,\delta) = G_V(\delta) - R_V(\delta),
\]
where
\begin{equation}\label{def:GVRV}
    \begin{split}
        G_V(\delta)
&=
\frac{1}{2\pi i} \int_{(1+\eps)} \Lscr(s,\delta)\Gamma(s-1)V^{s-1}\,ds,
\\
R_V(\delta)
&=
\frac{1}{2\pi i} \int_{(\sigma)} \Lscr(s,\delta)\Gamma(s-1)V^{s-1}\,ds.
    \end{split}
\end{equation}

The rest of this section is devoted to proving Lemma \ref{lemma:R_V} below,
which provides a bound for $R_V(\delta)$
when we sum over $\delta$ in a subset $\mathcal{M}$ of the Gaussian integers.
Following \cite[p.~725]{bykovskii_density_1997},
for a given finite set $\mathcal{M}\subseteq\ZZ[i]$, we define the quantities
\begin{equation}\label{2106:eq012}
    \begin{gathered}
        Q = Q(\Mcal) = 2 + \max_{m\in\Mcal} N(m),
        \\
        N_D(\Mcal) = \#\{m\in\Mcal:\; m\sim Dn^2\},\quad D\in\ZZ[i],\rule{0pt}{12pt}
        \\
        N(\Mcal) = \max_{D} N_D(\Mcal).\rule{0pt}{14pt}
    \end{gathered}
\end{equation}
In other words, $Q$ is essentially the maximal norm of the elements in $\Mcal$
and $N(\Mcal)$ is the size of the maximal intersection of $\Mcal$ with towers of the form $\{Dn^2,n\in\ZZ[i]\}$.
A crude upper bound for $N(\Mcal)$ is $\mathrm{Card}(\Mcal)$, but
$N(\Mcal)$ can in fact be much smaller.

\begin{lemma}\label{lemma:R_V}
    Let $\mathcal{M}$ be a finite subset of discriminants $\delta$ in $\ZZ[i]$ as in \S\ref{S2.2}
    and let $Q$ and $N(\Mcal)$ be as above. Then, for $1/2\leq\sigma<1$,
    \begin{equation}\label{lemma:R_V:eq}
        \sum_{\delta\in\mathcal{M}} |R_V(\delta)|
        \ll
        N(\Mcal) Q^{\frac{10(1-\sigma)}{3-\sigma}+\eps}
        +
        \mathrm{Card}(\Mcal) V^{\sigma-1} Q^\eps.
    \end{equation}
\end{lemma}

The result follows by an analysis of the number of zeros of Dirichlet
$L$-functions $L(s,\chi_D)$ near the line $\Re(s)=1$.
On the one hand, if there are no zeros in a given box, then one can deduce
a Lindel\"of-type bound for $L(s,\chi_D)$ in (almost) the full box.

\begin{lemma}\label{1706:lemma001}
    Let $1/2<\sigma_0<\sigma<1$ and $U>1$.
    Let $D\in\ZZ[i]$, $D$ not a perfect square, and
    assume that $L(s,\chi_D)$ has no zeros in the rectangle
    $[\sigma_0,1]\times[-U,U]$. Then we have the estimate
    \[
        L(s,\chi_D) \ll (DU)^\eps
    \]
    in the rectangle $[\sigma,1]\times[-U+1,U-1]$.
\end{lemma}

\begin{proof}
    See \cite[Lemma 1]{bykovskii_density_1997},
    where the case of $L$-functions over $\QQ$ is written out in detail.
    The proof uses standard analytic properties of $L$-functions
    and it generalises to $\QQ(i)$.
\end{proof}

On the other hand, if such zeros exist then pointwise bounds for $L(s,\chi_D)$ are not as strong.
However, we show that we can control the total number of zeros when
averaging over the family of characters $\chi_D$.

\begin{lemma}\label{lemma:huxley}
    Let $1/2\leq \sigma\leq 1$, $T>2$,
    and let $D\in\ZZ[i]$ be a generator of the discriminant of a quadratic extension of $\QQ(i)$.
    Denote by $N(\sigma,T,\chi_D)$ the number of zeros of $L(s,\chi_D)$
    in the rectangle $[\sigma,1]\times[-T,T]$.
    Then, for $Q\geq 1$, we have
    \[
        \sum_{N(D)\leq Q} N(\sigma,T,\chi_D)
        \ll
        Q^{\frac{10(1-\sigma)}{3-\sigma}+\eps} \, T^{\frac{4(1-\sigma)}{3-\sigma}+1+\eps}.
    \]
\end{lemma}

\begin{proof}
    In \cite[Theorem 2]{huxley_large_1971}
    Huxley proved
    a more general statement where he allows the $L$-functions to be twisted by
    a fixed Gr\"ossencharacter. We apply his result in the case when the twist
    is trivial.
    For a primitive character $\chi$ modulo $q\in\ZZ[i]$,
    consider the counting function of
    the number of zeros in a unit window defined as
    \[
        N(\sigma,U,U+1,\chi) = N(\sigma,U+1,\chi) - N(\sigma,U,\chi).
    \]
    Then, for every integer $l\geq 1$, we have the inequality
    \begin{equation}\label{1706:eq001}
        \begin{split}
            \sum_{N(q)\leq Q} \frac{N(q)}{\varphi(q)}
&\sum_{\chi\,(\mathrm{mod}\,q)} N(\sigma,U,U+1,\chi)
\\
&\ll
\Big( Q^{4l+4} + Q^{5l}U^{2l} \Big)^{(1-\sigma)/(2+l-2\sigma)} (\log QU)^{2+l(\sigma-1)}.
        \end{split}
    \end{equation}
    By positivity, the same inequality holds if we restrict
    the summation on the left to quadratic characters.
    Since $N(q)\geq \varphi(q)$,
    the lemma follows from \eqref{1706:eq001} by taking $l=4$
    and summing over all unit intervals up to $T$.
\end{proof}

We can now prove Lemma \ref{lemma:R_V}.

\begin{proof}[Proof of Lemma \ref{lemma:R_V}]
    By definition, we have
    \[
        R_V(\delta) = \frac{1}{2\pi i} \int_{(\sigma)} \Lscr(s,\delta)
        \Gamma(s-1) V^{s-1}\,ds.
    \]
    Using Lemma \ref{lemma:szmidt}, we write $\delta\sim Dl^2$,
    where $D$ generates the discriminant of the
	field extension associated to $\delta$,
    and factor $\Lscr(s,\delta)=T_l^{(D)}(s)L(s,\chi_D)$,
    which gives
    \begin{equation}\label{1706:eq002}
        R_V(\delta) = \frac{1}{2\pi i} \int_{(\sigma)} T_l^{(D)}(s)L(s,\chi_D)
        \Gamma(s-1) V^{s-1}\,ds.
    \end{equation}
    Next we sum over $\Mcal$. Set $T=1 + (\log N(D))^{2}$,
    and split the sum into two parts according to whether
    $\delta\in\Mcal_1$ or $\delta\in\Mcal_2$,
    where
    \[
        \begin{split}
            \Mcal_1 &= \{\delta\in\Mcal: \; L(s,\chi_D) \text{ has a zero in }[\sigma,1]\times[-T,T]\},
            \\
            \Mcal_2 &= \{\delta\in\Mcal: \; L(s,\chi_D) \text{ has no zeros in }[\sigma,1]\times[-T,T]\}.
        \end{split}
    \]
    If $\delta\in\Mcal_1$, we use Lemma \ref{lemma:szmidt}
    along with the estimates
    $L(s,\chi_D)\ll N(\delta)^\eps$ and
    $T_l^{(D)}(s)\ll N(\delta)^\eps$, for $1\leq \Re(s)\leq 1+\eps$, to bound
    \[
        R_V(n^2-4) = - \Lscr(1,n^2-4) + G_V(n^2-4) \ll Q^\eps
    \]
    with $Q$ as in the statement of the lemma.
    Therefore,
    \[
        \sum_{\delta\in\Mcal_1} R_V(\delta) \ll Q^\eps \mathrm{Card}(\Mcal_1).
    \]
    The last cardinality, in view of Lemma \ref{lemma:huxley}, is at most
    \[
        \mathrm{Card}(\Mcal_1) \ll N(\Mcal) Q^{\frac{10(1-\sigma)}{3-\sigma}+\eps}.
    \]
    Combining the two inequalities above gives the first part of the bound in \eqref{lemma:R_V:eq}.
    As for the sum over $\delta\in\Mcal_2$, we use \eqref{1706:eq002} to estimate $R_V(\delta)$.
    The tails of the integral over $|\Im(s)|\geq T-1$ are bounded by
    using the exponential decay of the Gamma function
    and standard polynomial bounds on $L(s,\chi_D)$, which yields
    \[
        \underset{\substack{\Re(s)=\sigma\\|\Im(s)|\geq T-1}}{\int}
        T_l^{(D)}(s)L(s,\chi_D)\Gamma(s-1)V^{s-1}\,ds
        \ll
        V^{\sigma-1}.
    \]
    For the integral over $|\Im(s)|<T-1$, we use Lemma \ref{1706:lemma001} to bound
    $L(s,\chi_D)\ll Q^\eps$, and obtain
    \begin{equation}\label{0409:eq001}
        \underset{\substack{\Re(s)=\sigma\\|\Im(s)|\leq T-1}}{\int}
        T_l^{(D)}(s) L(s,\chi_D) \Gamma(s-1) V^{s-1}\,ds
        \ll
        V^{\sigma-1}Q^\eps.
    \end{equation}
    Thus we deduce that the sum over $\delta\in\Mcal_2$ contributes at most
    \[
        \sum_{\delta\in\Mcal_{2}}R_{V}(\delta)
        \ll \mathrm{Card}(\Mcal_2)V^{\sigma-1}Q^\eps
        \ll \mathrm{Card}(\Mcal)V^{\sigma-1}Q^\eps.
    \]
    This gives the second term in \eqref{lemma:R_V:eq}
    and concludes the proof of Lemma \ref{lemma:R_V}.
\end{proof}

\section{Proof of Theorems~\ref{intro:theorem1} and~\ref{theorem:conditional}}
\label{S3}

\subsection{A theorem of Wu and Z\'abr\'adi}
Our starting point in proving Theorem \ref{intro:theorem1}
is a formula that relates the counting function $\Psi_\Gamma(X)$
to the $L$-functions $\Lscr(s,\delta)$ introduced in the previous section.
Such a formula has been proved recently
by Wu and Z\'abr\'adi \cite{wu_kuznetsov-bykovskiis_2019}.
In our notation it can be stated as follows.

\begin{theorem}[{\cite[Theorem
    1.4]{wu_kuznetsov-bykovskiis_2019}}]\label{thm:wuzabradi}
    Let $X>2$. There is an absolute constant $C$ such that
    \begin{equation}\label{2106:eq001}
        \Psi_\Gamma(X) = C \sum_{n} \sqrt{N(n^2-4)} \Lscr(1,n^2-4) + O(1),
    \end{equation}
    where the sum is restricted to $n\in\ZZ[i]$ satisfying the condition
    \begin{equation}\label{2106:eq010}
        1<\max_{\pm}\, N\Big(\frac{n\pm\sqrt{n^2-4}}{2}\Big) \leq X.
    \end{equation}
\end{theorem}

The result in \cite[Theorem 1.4]{wu_kuznetsov-bykovskiis_2019}
is in fact more general as they allow number fields other than
$\QQ(i)$.

Due to many differences in notation, we briefly explain how to arrive
at Theorem~\ref{thm:wuzabradi} from their statement.
First, the definition of $\Psi_\Gamma(X)$ in \cite{wu_kuznetsov-bykovskiis_2019}
differs from ours as they consider the function
\[
    \Psi_\Gamma^*(X)  =\sum_{N(P)\leq X} \frac{\Lambda_\Gamma(N(P))}{\Ecal(P)}.
\]
Here $\Ecal(P)$ is a positive integer equal to a fixed constant,
say, $C_1$, except for finitely many conjugacy classes
(see \cite[\S5.2]{elstrodt_groups_1998} and \cite[(5.5)]{sarnak_arithmetic_1983}).
Therefore, we have
\[
    \Psi_\Gamma(X) = C_1 \Psi_\Gamma^*(X) + O(1),
\]
which explains the error term in \eqref{2106:eq001}.
Next, \cite[Theorem 1.4 (1)]{wu_kuznetsov-bykovskiis_2019} gives the identity
\begin{equation}\label{2106:eq002}
    \Psi_\Gamma^*(X) = \sum_{n} |d_{n^2-4}|_\infty^{1/2} L_\Gamma(1,n^2-4),
\end{equation}
where the sum is over $n\in\ZZ[i]$ subject to the restriction
\begin{equation}\label{2106:eq003}
    \max\Bigl\{\Big|\frac{n+\sqrt{n^2-4}}{2}\Big|_{\infty},
            \Big|\frac{n-\sqrt{n^2-4}}{2}\Big|_{\infty}\Bigr\} \leq X.
\end{equation}
In both \eqref{2106:eq002} and \eqref{2106:eq003}, the notation $|x|_\infty$
refers to the absolute value of $x$ at the complex place,
i.e.~the norm $N(x)$.\footnote{There is a typo in~\cite{wu_kuznetsov-bykovskiis_2019},
    where the subscript $\infty$ is missing from both of the absolute values.
We thank H.~Wu for clarifying their result to us.}

The $L$-function $L_\Gamma(s,\delta)$ in \eqref{2106:eq002} is by
\cite[Theorem 1.4 (2)]{wu_kuznetsov-bykovskiis_2019}
of the form
\[
    L_\Gamma(s,\delta) = C_2 P_\Gamma(s) L(s,\chi_{d_\delta}),
\]
where $C_2$ is a constant that depends only on the base field
and on the group under consideration.
Moreover, $\delta$ is factored as $\delta\sim d_\delta l^2$,
where $d_\delta$ generates the discriminant of the field extension $\QQ(i)(\sqrt{\delta})$,
as in \S\ref{S2.2}, and the factor $P_\Gamma(s)$ is a Dirichlet polynomial
that can be written as a product over primes dividing $(\delta/d_\delta)$.
By \cite[(4.7)]{wu_kuznetsov-bykovskiis_2019} we see that at each prime $\pfrak$
we have a factor $N(\pfrak^{l_\pfrak/2})$, where $\pfrak^{2l_\pfrak}$
is the exact power of $\pfrak$ dividing $(\delta/d_\delta)$.
Therefore, by collecting these factors we get
\[
    P_{\Gamma}(s) = N(\delta/d_\delta)^{1/4} \, P_\Gamma^*(s),
\]
for some other Dirichlet polynomial $P_\Gamma^*(s)$.
Comparing this with $T_l^{(d_\delta)}(s)$ from section~\ref{S2.2},
when $l$ is a prime power
(cf.~\cite[(7)]{soundararajan_prime_2013} for the rational case),
one can further deduce that
\[
    P_\Gamma^*(s) = T_l^{(d_\delta)}(s) N(\delta/d_\delta)^{\frac{s}{2}-\frac{1}{4}}.
\]
Hence, we obtain
\[
    |d_\delta|_\infty^{1/2} L_\Gamma(1,\delta)
    =
    C_2 \sqrt{N(\delta)} T_l^{(d_\delta)}(1) L(1,\chi_{d_\delta})
    =
    C_2 \sqrt{N(\delta)}\Lscr(1,\delta).
\]
Setting $C=C_1C_2$ and evaluating at $\delta=n^2-4$ we obtain \eqref{2106:eq001}.

\subsection{Application of the auxiliary lemmas}
Once \eqref{2106:eq001} is established, we proceed as follows.
First, we replace the condition \eqref{2106:eq010} by a simpler one
at the cost of an admissible error term. Write
\[
    z = \frac{n+\sqrt{n^2-4}}{2} = re^{i\vartheta},
    \qquad
    z^{-1} = \frac{n-\sqrt{n^2-4}}{2} = r^{-1}e^{-i\vartheta}.
\]
Up to interchanging the roles of $z$ and $z^{-1}$, we can assume that $r>1$.
We can then express $n$ in terms of $z$ and $z^{-1}$ as
\[
    n = z + z^{-1} = re^{i\vartheta} + r^{-1}e^{-i\vartheta}.
\]
Setting $X_0=\sqrt{X}+1/\sqrt{X}$, an easy computation shows that \eqref{2106:eq010}
is equivalent to
\[
    |n|^2 + 4\sin^2(\vartheta) \leq X_0^2 = X + 2 + \frac{1}{X}.
\]
Therefore, the condition \eqref{2106:eq010} can be replaced by $N(n)\leq X$
up to miscounting $O(X^\eps)$ points in the annulus $N(n)=X+O(1)$.
Bounding $\Lscr(1,n^2-4)\ll N(n)^\eps$
and approximating $N(n^2-4)=N(n^2)+O(N(n))$
we can thus write
\[
    \Psi_\Gamma(X)
    =
    \;C\!\!\!\!\!\sum_{N(n) \leq X} N(n) \Lscr(1,n^2-4) + O(X^{1+\eps}).
\]
It follows that in intervals of the form $[X,X+Y]$,
with $Y$ as in the statement of Theorem \ref{intro:theorem1}, we have
\begin{equation}\label{2106:eq013}
    \Psi_\Gamma(X+Y) - \Psi_\Gamma(X)
    =
    \;\;C\!\!\!\!\!\!\!\!\!\!
    \sum_{X<N(n) \leq X+Y}\!\!\!\!\!\!\!
    N(n) \Lscr(1,n^2-4) + O(X^{1+\eps}).
\end{equation}
At this point we approximate $\Lscr(1,n^2-4)$ by an absolutely
convergent Dirichlet series as anticipated in \S\ref{S2.3}.
For $V>0$, we write
\begin{equation}\label{2106:eq014}
    \Lscr(1,n^2-4) = G_V(n^2-4) - R_V(n^2-4),
\end{equation}
where $G_V(n^2-4)$ and $R_V(n^2-4)$ are as in \eqref{def:GVRV}.
Let
\[
    \Mcal = \{n^2-4\in\ZZ[i]:\; X<N(n)\leq X+Y\} \subseteq \ZZ[i].
\]
In particular, notice that $\mathrm{Card}(\Mcal)\ll YX^\eps$.
Moreover, if $Q=Q(\Mcal)$ and $N(\Mcal)$ are defined
as in \eqref{2106:eq012}, we then have
\begin{equation}\label{2506:eq001}
    Q \ll X^2, \qquad N(\Mcal) \ll X^\eps.
\end{equation}
The first inequality is immediate.
Concerning the second estimate we recall that,
by a result of Sarnak \cite[pp.~275--276]{sarnak_arithmetic_1983},
the solutions in $\ZZ[i]$ to the Pell equation $n^2-Dl^2=4$, when $D$
is not a square,
are all powers of a fundamental solution $\varepsilon_D=(n_0+\sqrt{D}l_0)/2$.
We are interested in those with $|\eps_D|>1$.
Since $\varepsilon_D^{-1}=(n_0-\sqrt{D}l_0)/2$, we deduce that
\[
    |\varepsilon_D-\varepsilon_D^{-1}|=|\sqrt{D}l_0|\geq \alpha>1,
\]
which in turn implies $|\varepsilon_D|\geq \alpha'>1$, uniformly in $D$.
Consequently, the number of solutions of size less than a given quantity $X$
is at most $O(\log X)$, uniformly in~$D$.
This proves the second inequality in \eqref{2506:eq001}.

We now go back to \eqref{2106:eq013} and use \eqref{2106:eq014}
to replace $\Lscr(1,n^2-4)$. We bound the sum $R_V(n^2-4)$
in the interval $X\leq N(n)\leq X+Y$
by using Lemma \ref{lemma:R_V},
and obtain
\[
    \begin{split}
        \Psi_\Gamma(X+Y) - \Psi_\Gamma(X)
        =
        \;\;\; C
        \!\!\!\!\!\!\!\!\!\!\!\!
        \sum_{X\leq N(n)\leq X+Y}
        \!\!\!\!
&N(n) G_V(n^2-4)
\\
&+
O\bigl(X^{1+\frac{20(1-\sigma)}{3-\sigma}+\eps}
    + X^{1+\eps}YV^{\sigma-1}\bigr),
    \end{split}
\]
for any $1/2\leq\sigma<1$.
In the main term we expand $G_V(n^2-4)$ into a Dirichlet series
and write
\[
    \sum_{X\leq N(n)\leq X+Y} \!\!\!\!\!\! N(n) G_V(n^2-4)
    =
    \sum_{q\neq 0} \frac{e^{-N(q)/V}}{N(q)}
    \!\!\! \sum_{X\leq N(n)\leq X+Y} \!\!\!\!\!\! N(n)\lambda_q(n^2-4).
\]
The summation over $n$ can be performed by parts and by using Lemma \ref{1606:lemma001},
which leads to
\begin{multline*}
    \sum_{X\leq N(n)\leq X+Y} \!\!\!\!\!\!  N(n) G_V(n^2-4)
    =
    \pi\left(XY + \frac{Y^2}{2}\right)
    \sum_{q\neq 0} \frac{e^{-N(q)/V}}{N(q)}
    \sum_{q_1^2q_2=q}\frac{\mu(q_2)}{N(q_2)}\\
    +
    O\bigl(\min\{X^{4/3}V^{1/3+\epsilon},
    X^{1+\eta+\eps}V^{1-\eta+\eps}\}\bigr).
\end{multline*}
The sum over $q$ on the first line gives
\[
    \begin{split}
        \sum_{q\neq 0} \frac{e^{-N(q)/V}}{N(q)}
        \sum_{q_1^2q_2=q}\frac{\mu(q_2)}{N(q_2)}
&=
\frac{1}{2\pi i} \int_{(1+\eps)}
\frac{\zeta_{\QQ(i)}(2+2s)}{\zeta_{\QQ(i)}(2+s)}\Gamma(s)V^s\,ds
\\
&=
1 + O(V^{-1/2+\eps}).
    \end{split}
\]
In summary, we have proved that
\begin{equation}\label{0109:eq001}
    \begin{split}
        \Psi_\Gamma(X+Y) - &\Psi_\Gamma(X) =
        \pi\, C \biggl(XY +\frac{Y^2}{2}\biggr)
        + O\bigl(X^{1+\frac{20(1-\sigma)}{3-\sigma}+\eps}
        \bigr)\\
                                           &+ O\bigl(
            X^{1+\eps}YV^{\sigma-1} +
            \min\{X^{4/3}V^{1/3+\epsilon},
            X^{1+\eta+\eps}V^{1-\eta+\eps}\}
        \bigr).
    \end{split}
\end{equation}

\begin{remark}
Note that the identity \eqref{2106:eq013},
by the trivial bound $\Lscr(1,n^2-4)\ll |n|^\eps$
and the observation that the number of Gaussian integers
with given norm is $O(X^\eps)$,
immediately implies the estimate
\begin{equation}\label{2910:eq002}
\Psi_\Gamma(X+Y) - \Psi_\Gamma(X) \ll X^{1+\eps}Y,
\end{equation}
for every $X\gg 1$ and $Y\geq 1$.
This is analogous to \cite[Lemma 4]{iwaniec_1984},
and will be used in section~\ref{S4}.
\end{remark}

\subsection{Optimisation of parameters}
We now optimise the parameters $V$ and $\sigma$ in \eqref{0109:eq001}.
Consider the first term in the minimum.
Balancing this with the term $X^{1+\epsilon}YV^{\sigma-1}$ gives
\[V=(X^{-1}Y^{3})^{\frac{1}{4-3\sigma}}.\]
Hence the error in~\eqref{0109:eq001} is bounded by
\begin{equation}\label{eq:optoeq}
    O\bigl(X^{1+\frac{20(1-\sigma)}{3-\sigma}+\epsilon}
    +XY(X^{-1}Y^{3})^{\frac{\sigma-1}{4-3\sigma}+\epsilon}\bigr).
\end{equation}
Then we optimise $\sigma$ according to the relative size of $Y$ and $X$.
Recalling that $Y=X^\nu$, we choose $\sigma\in[\frac{1}{2},1)$ such that
\begin{equation}\label{eq:sigmaeq}
    \frac{20(1-\sigma)}{3-\sigma}-\nu=\frac{(\sigma-1)(3\nu-1)}{4-3\sigma}.
\end{equation}
This is possible since for $\sigma=1$ and $\sigma=1/2$, the sign
of the left and right-hand sides in~\eqref{eq:sigmaeq} are in reverse order.
For this particular value of $\sigma$, both terms in~\eqref{eq:optoeq} give
\[O((XY)X^{-\beta(\nu)+\epsilon}),\]
where $\beta$ is given by
\begin{equation}\label{eq:beta}
    \beta(\nu) = \frac{(1-\sigma)(3\nu-1)}{4-3\sigma},
\end{equation}
for $1/3<\nu\leq 1$.
Combining this with \eqref{0109:eq001} we conclude that
\[
    \Psi_\Gamma(X+Y)
    -
    \Psi_\Gamma(X)
    =
    \pi\,C \left(XY+\frac{Y^2}{2}\right)
    +
    O((XY)\,X^{-\beta(\nu)+\eps}).
\]
Evaluating this for $Y=X$ and comparing with the asymptotic
$\Psi_\Gamma(X)\sim\frac{1}{2}X^2$ (see~\eqref{eq:sarnak}),
we also deduce that $\pi\,C=1$.

Now, for the second term in the minimum in~\eqref{0109:eq001}
we instead balance with
\[
    V = (YX^{-\eta})^{\frac{1}{2-\eta-\sigma}},
\]
so that the error in \eqref{0109:eq001} becomes
\begin{equation}\label{0809:eq001}
    O\bigl(
        X^{1+\frac{20(1-\sigma)}{3-\sigma}+\eps}
        +
        X^{1+\eta+\eps}(YX^{-\eta})^{\frac{1-\eta+\eps}{2-\eta-\sigma}}
    \bigr).
\end{equation}
We need $\sigma\in[\frac{1}{2},1)$ such that
\begin{equation}\label{eq:sigmaeq2}
    1+\frac{20(1-\sigma)}{3-\sigma}
    =
    1+\eta+
    (\nu-\eta)\frac{1-\eta}{2-\eta-\sigma}.
\end{equation}
As before, such a $\sigma$ is guaranteed to exists by considering the sign of
both sides of~\eqref{eq:sigmaeq2} at $\sigma=1/2$ and $\sigma=1$.
Hence the error in~\eqref{0809:eq001} is then
\[
    O(X^{1+\nu-\alpha(\nu,\eta)+\eps})
    =
    O((XY)\,X^{-\alpha(\nu,\eta)+\eps}),
\]
where
\begin{equation}\label{0809:eq002}
    \alpha(\nu,\eta) = (\nu-\eta)\frac{1-\sigma}{2-\eta-\sigma}.
\end{equation}
This concludes the proof of Theorem~\ref{intro:theorem1}.

\subsection{Application of the subconvexity bound}
In this section we prove Theorem~\ref{theorem:conditional}.
To do this, we imitate the proof of Theorem~\ref{intro:theorem1}, but
in Lemma~\ref{lemma:R_V} we instead take $\Mcal_1=\emptyset$,
$\Mcal_2=\Mcal$ and shift the integral in~\eqref{0409:eq001} to
$\sigma=1/2$ and use the subconvexity estimate~\eqref{intro:subconv}.
It follows that we can replace the bound~\eqref{lemma:R_V:eq} by
\begin{equation}\label{S4:eq002}
    \sum_{\delta\in\mathcal{M}} |R_V(\delta)|
    \ll
    \mathrm{Card}(\Mcal) V^{-1/2} Q^{\theta+\eps}.
\end{equation}
Since $Q\ll X^{2}$ (see~\eqref{2506:eq001}), this means that the remainder
in~\eqref{0109:eq001} becomes
\[O(X^{1+2\theta+\epsilon}YV^{-1/2} + X^{4/3}V^{1/3+\epsilon}).\]
We balance this with
\(V= (X^{2\theta-1/3}Y)^{6/5},\)
which yields
\(O(X^{(4\theta+6)/5+\epsilon}Y^{2/5}),\)
as required. The second bound~\eqref{eq:shortintervalgauss} follows by
using~\eqref{S4:eq002} with $\theta=0$ and by using the second term from the
minimum in Lemma~\ref{1606:lemma001} with $\eta=1/4$.


\section{Smooth Explicit Formula}\label{S4}

A standard way to obtain estimates for the prime geodesic theorem
is to relate $\Psi_{\Gamma}$ to an exponential sum over the spectral parameters
$r_{j}$, known as \emph{explicit formulae}.
In three dimensions, such a formula was proved by
Nakasuji~\cite{nakasuji_2000,nakasuji_2001}.
Let $X\gg 1$, $T\geq 1$ and suppose that $T<X^{1/2}$. Then, her explicit
formula says that
\begin{equation}\label{2810:eq001}
\Psi_\Gamma(X) = \frac{1}{2}X^2 + 2\Re\biggl(\sum_{0<r_j\leq T} \frac{X^{1+ir_j}}{1+ir_j}\biggr)
+
O\biggl(\frac{X^{2}}{T}\log X\biggr).
\end{equation}
In fact, Nakasuji's proof shows that there are also secondary terms
that contribute
$O(XT\log T+T^2)$ (see \cite[(5.13)]{nakasuji_2000}).
Clearly, these terms get absorbed into the error in~\eqref{2810:eq001} if $T<X^{1/2}$,
so that the optimal bound is $O(X^{3/2}\log X)$.

In this paper we instead consider a smoothed version of~\eqref{2810:eq001},
which allows us to relax the conditions on $T$ and, in particular, to
break the barrier $O(X^{3/2+\epsilon})$. We note that in two dimensions such a
smooth explicit formula is not needed as the pointwise version proved by
Iwaniec~\cite{iwaniec_1984} is optimal.
Let $k$ be a smooth, real-valued function
with compact support on $(Y, 2Y)$. Moreover,
assume that $k$
is of unit mass and satisfies
$\int|k^{(j)}(u)|\,du\ll_{j}Y^{-j}$ for all $j\geq 0$. Define
\begin{equation}\label{2810:eq002}
    \Psi_{\Gamma}(X, k) = \int_{Y}^{2Y}\Psi_{\Gamma}(X+u)k(u)\,du.
\end{equation}
We then have the following explicit formula for $\Psi_{\Gamma}(X, k)$
(cf.~\cite[\S10.3]{hejhal_1983}, \cite[Theorem~4.7]{nakasuji_2004}).

\begin{lemma}\label{2810:lemma001}
    Let $T,X,Y\gg 1$, with $T,Y\leq X$
    and
    $TY>X^{1+\xi}$ for some $\xi>0$.
    Then
    \begin{equation}\label{2810:eq008}
        \begin{split}
            \Psi_\Gamma(X,k)
            =
            \int_{Y}^{2Y} \biggl(\frac{1}{2}(X+u)^2 + 2\Re\biggl(\sum_{0<r_j\leq T}
            \frac{1}{1+ir_j}(X+u)^{1+ir_j}\biggr)\biggr) k(u)\, du
            \\
            +
            O\biggl(\frac{X^{2+\eps}}{T} + \frac{X^{2+\eps}}{Y^{2}} + X^{1+\eps}\biggr).
        \end{split}
    \end{equation}
\end{lemma}

Before giving a proof of Lemma~\ref{2810:lemma001},
we recall the definition of the Selberg zeta function
and its logarithmic derivative.
For $s\in\CC$ with $\Re(s)>2$, the Selberg zeta function
is defined as
\[
    Z(s) = \prod_{\{P_0\}} \prod_{(k,l)}(1- {a(P)\rule{0pt}{9pt}}^{-2k}\,{\overline{a(P)\!}}^{\,-2l}N(P_0)^{-s}),
\]
where the outer product runs over primitive hyperbolic
and loxodromic conjugacy classes of $\Gamma$,
and the inner product runs over all the pairs
of non-negative integers such that $k\equiv l$~mod~$m(P_0)$,
where $m(P_0)$ denotes
the order of the torsion of the centraliser of $P_0$
(see \cite[p.~206, Definition 4.1]{elstrodt_groups_1998}).
$Z(s)$ extends to
a meromorphic function on $\CC$ with a functional equation
relating the values at $s$ and $2-s$. The Selberg zeta function has non-trivial
spectral zeros at each $s_{j}=1+ir_{j}$ and $\bar{s}_{j}$
(for $\lambda_{j}=s_{j}(2-s_{j})$).
Therefore, the sums in~\eqref{2810:eq001} and \eqref{2810:eq008}
correspond to sums over $s_j$.
In addition, $Z$ also vanishes at the non-trivial zeros $\rho_{j}$ of the
Dedekind zeta function $\zeta_{\QQ(i)}$, which lie to the left of the critical
line $\Re(s)=1$~\cite[\S7.4]{sarnak_arithmetic_1983}.
Since $\Gamma$ has no small non-trivial
eigenvalues (i.e.~$s_{j}\not\in[1,2)$) \cite[Proposition~7.6.2]{elstrodt_groups_1998}
it means that, apart from the trivial zero at $s=2$,
$Z(s)$ is non-zero for $\Re(s)>1$, i.e.~we know the analogue of the Riemann
hypothesis for $Z$. For a complete description of the zeros and singularities of
$Z$ see~\cite[\S4]{gangolli_1980}.

By~\cite[p.~208, Lemma 4.2]{elstrodt_groups_1998},
the logarithmic derivative of $Z$ is given, again for $\Re(s)>2$, by
\[
    \frac{Z'}{Z}(s) = \sum_{\{P\}} \frac{N(P)\Lambda_\Gamma(N(P))}{m(P)|a(P)-a(P)^{-1}|^2} N(P)^{-s},
\]
where the sum runs over all hyperbolic and loxodromic
conjugacy classes of $\Gamma$, and
$a(P)$, $a(P)^{-1}$ are the eigenvalues of $P$ with $|a(P)|>1$.
Recalling that $N(P)=|a(P)|^2$, and that $m(P)\neq 1$
only for finitely many classes (see \cite[p.~224]{elstrodt_groups_1998}),
we deduce that
\[
    \frac{N(P)\log(N(P_0))}{m(P)|a(P)-a(P)^{-1}|^2}
    =
    \Lambda_\Gamma(N(P)) + O(N(P)^{-1+\eps}).
\]

\begin{proof}[Proof of Lemma \ref{2810:lemma001}]
    We follow \cite[\S5]{nakasuji_2000}. Let $\eps\in(0,1)$, and let $c=2+\eps$.
    By a standard application of Perron's formula we can write
    \[
        \frac{1}{2\pi i}\int_{c-iT}^{c+iT} \frac{Z'}{Z}(s) \frac{X^s}{s} ds
        =
        \Psi_\Gamma(X) + R,
    \]
    where
    \[
        R \ll X^{1+\eps} + X^c\sum_{\{P\}} \frac{\Lambda_\Gamma(N(P))}{N(P)^c}\min\biggl(1,\frac{1}{T|\log (X/N(P))|}\biggr).
    \]
    We split the sum at $|N(P)-X|<X/2$. Using the upper bound $\Psi_\Gamma(X)\ll X^2$
    and the second term in the minimum,
    we can bound the beginning and the tail of the sum by $O(X^{c}T^{-1})$.
    Furthermore, by decomposing into intervals of length $2X/T$,
    the remaining part of the series contributes
    \[
        \sum_{|N(P)-X|\leq X/T} \Lambda_\Gamma(N(P))
        \;\;+
        \sum_{\substack{-T/4\leq k\leq T/4\\k\neq 0}}\;
        \sum_{\{P\}\in I_k} \frac{\Lambda_\Gamma(N(P))}{T|\log(X/N(P))|},
    \]
    where $I_k=\{\{P\}:|N(P)-X-2kX/T|\leq X/T\}$.
    By the short interval estimate \eqref{2910:eq002}
    we deduce that the first sum is bounded by $O(X^{2+\eps}T^{-1})$,
    and the second sum is bounded by
    \[
        \frac{X^{2+\eps}}{T}\sum_{k=1}^{T/4} \frac{1}{|k|}
        \ll
        \frac{X^{2+\eps}}{T}.
    \]
    In other words, we have
    \[
        \frac{1}{2\pi i}\int_{c-iT}^{c+iT} \frac{Z'}{Z}(s) \frac{X^s}{s} ds
        =
        \Psi_\Gamma(X) + O\biggl(\frac{X^{2+\eps}}{T}\biggr).
    \]
    Next we move the line of integration to the left of the critical strip
    and we pick up the poles of $Z'/Z$.
    In order to do so, we suppose that $T$ is not the ordinate of a zero of
    $Z(s)$.
    This leads to the identity
    \begin{equation}\label{2810:eq007}
        \begin{split}
            \Psi_\Gamma(X) + O\biggl(\frac{X^{2+\eps}}{T}\biggr)
&=
\frac{1}{2} X^2 + 2\Re\biggl(\sum_{0<r_j\leq T} \frac{X^{s_j}}{s_j}\biggr)
+ 2\Re\biggl(\sum_{0<\gamma_j\leq T} \frac{X^{\rho_j}}{\rho_j}\biggr)
\\
&\pm
\frac{1}{2\pi i}\int_{\Cscr_1^\pm} \frac{Z'}{Z}(s) \frac{X^s}{s} ds
+
\frac{1}{2\pi i}\int_{\Cscr_2} \frac{Z'}{Z}(s) \frac{X^s}{s} ds,
        \end{split}
    \end{equation}
    where we have written $\rho_j=\beta_j+i\gamma_j$, and the contours are given by
    \[
        \Cscr_1^\pm = [-\epsilon\pm iT,c\pm iT], \quad
        \Cscr_2 = [-\epsilon-iT,-\epsilon+iT].
    \]
    In~\eqref{2810:eq007} we let $X\mapsto X+u$ and integrate against $k$.
    It remains to show that~the sum over $\rho_j$ and the
    integrals get absorbed into the error
    in~\eqref{2810:eq008}.
    The sum is easily bounded by $O(X^{1+\eps})$ since
    $\#\{\rho_j:|\gamma_j|\leq T\}\ll T\log T$.
    Denote by $\Ical_1^\pm$ and $\Ical_2$ the integrals
    over $\Cscr_1^\pm$ and $\Cscr_2$ in \eqref{2810:eq007}.
    Observe that by repeated integration by parts we have,
    for every $l\geq 0$,
    \begin{equation}\label{byparts}
        \int_{Y}^{2Y} (X+u)^s k(u)\, du \ll_l \frac{X^{\Re(s)+l}}{|sY|^l}.
    \end{equation}
    We also need the fact that, for all $T\gg 1$, there exists $\tau\in [T, T+1]$
    such that (see~\cite[(5.7)]{nakasuji_2004}, \cite[(3.10)]{nakasuji_2000},
    \cite[(25)]{iwaniec_1984})
    \begin{equation}\label{eq:zint}
        \int_{0}^{2}\biggl\lvert\frac{Z'}{Z}(\sigma+i\tau)\biggr\rvert\,d\sigma\ll
        T^{2}\log T.
    \end{equation}
    Combining~\eqref{byparts} and~\eqref{eq:zint}, we get (by changing
    $T$ by a bounded amount)
    \[
        \Ical_1^\pm
        =
        \frac{1}{2\pi i} \int_{\Cscr_1^\pm} \frac{Z'}{Z}(s)
        \int_{Y}^{2Y} (X+u)^{s}k(u) \,du \frac{ds}{s}
        \ll
        \frac{X^{2+l+\eps}T^{2+\epsilon}}{|TY|^{l}} \ll 1,
    \]
    where the last inequality follows from the assumption
    $TY>X^{1+\xi}$, and on taking $l$ sufficiently large.
    To the left of the critical strip, we have
    (see~\cite[(5.5)]{nakasuji_2004}, \cite[(3.8)]{nakasuji_2000},
    \cite[(24)]{iwaniec_1984})
    \[\frac{Z'}{Z}(-\epsilon+it)\ll |t|^{2}+1.\]
    Thus we can bound $\Ical_2$ as
    \[
        \Ical_2
        \ll
        X^{-\epsilon} \int_{|s|<X^{1+\xi}/Y} |s|\, |ds|
        +
        \frac{X^{l-\epsilon}}{Y^{l}} \int_{|s|>X^{1+\xi}/Y} \frac{|ds|}{|s|^{l-1}}
        \ll
        \frac{X^{2+\eps}}{Y^2} + 1.
    \]
    Finally, the assumption on $T$ can be dropped by changing $T$ by a bounded quantity.
    This amounts to extending the sum in \eqref{2810:eq008} to $T\leq r_j\leq T+O(1)$.
    In view of~\eqref{byparts}, and recalling that the number of such terms is $O(T^2)$ by the Weyl law,
    this additional contribution gets absorbed into the error.
\end{proof}

%
%

\section{Recovering Pointwise Bounds}
In this section we prove Corollaries~\ref{cor:uncond},~\ref{cor:11/7},~and~\ref{cor:34/23}.
The argument is essentially identical to that
of~\cite[\S3--4]{soundararajan_prime_2013}, but we reproduce it here for the
sake of completeness.
The main idea is to consider the same smoothed $\Psi_{\Gamma}$ as in the
previous section. We can then combine the smooth explicit formula and
short interval bounds to recover the unsmoothed function.
To that end, let $k$ be as before
and consider the function $\Psi_\Gamma(X,k)$ defined in \eqref{2810:eq002}.
Clearly, we have
\begin{equation}\label{eq:smooth}
    \Psi_{\Gamma}(X) = \Psi_{\Gamma}(X, k)
    -\int_{Y}^{2Y}(\Psi_{\Gamma}(X+u)-\Psi_{\Gamma}(X))k(u)\,du.
\end{equation}
The integral in~\eqref{eq:smooth} can be treated with
Theorem~\ref{intro:theorem1}.
To estimate $\Psi_{\Gamma}(X, k)$,
we use the smooth explicit formula proved in Lemma \ref{2810:lemma001}.
Let $T,Y\gg 1$ with $T,Y\leq X$, and assume that $TY>X^{1+\xi}$ for some $\xi>0$.
Then, Lemma \ref{2810:lemma001} gives
\[
    \begin{split}
        \Psi_{\Gamma}(X, k) =
        \int_{Y}^{2Y}\biggl(\frac{1}{2}(X+u)^{2}+2\Re\biggl(\sum_{0<r_{j}\leq
        T}\frac{(X+u)^{1+ir_{j}}}{1+ir_{j}}\biggr)\biggr)k(u)\,du
        \\
        +
        O(X^{2+\eps}T^{-1} + X^{2+\eps}Y^{-2} + X^{1+\eps}).
    \end{split}
\]
If we pick $T=X$ and $Y\geq X^{1/2}$,
we may then write
\begin{equation}\label{eq:psixkasymp}
    \Psi_{\Gamma}(X, k) = \frac{1}{2}\int_{Y}^{2Y} (X+u)^{2}k(u)\,du
    +
    2\Re(E(X, k)) + O(X^{1+\eps}),
\end{equation}
where
\begin{equation}\label{eq:exkdef}
    E(X, k)= \!\!\!\!\!\sum_{0<r_{j}\leq X}
    \frac{1}{1+ir_{j}}
    \int_{Y}^{2Y}(X+u)^{1+ir_{j}}k(u)\,du.
\end{equation}
The sum in $E(X, k)$ can be truncated further at $X^{1+\xi}/Y$.
To see this, we again integrate by parts $l$ times (as in~\eqref{byparts}) and
get
\[
    \int_{Y}^{2Y}(X+u)^{1+ir_{j}}k(u)\,du
    \ll_{l}\frac{X^{1+l}}{|1+ir_{j}|^{l}Y^{l}}.
\]
Therefore, by choosing a suitably large $l$
and recalling that $\#\{r_j\leq T\}\ll T^3$ by the Weyl law,
we deduce that
\begin{equation}\label{eq:tailsum}
    \sum_{X^{1+\xi}/Y< r_{j}\leq X}
    \frac{1}{1+ir_{j}}\int_{Y}^{2Y}(X+u)^{1+ir_{j}}k(u)\,du
    \ll X^{1+\epsilon}.
\end{equation}
For the remaining part of the sum, we need to
understand the spectral exponential sum defined as
\begin{equation}\label{eq:stx}
    S(T, X)=\sum_{0<r_{j}\leq T}X^{ir_{j}}.
\end{equation}
We appeal to the following bound
proved in~\cite[Theorem~3.2]{balkanova_prime_2018-1}
\begin{equation}\label{eq:stxbound}
    S(T, X)\ll T^{2+\epsilon}X^{1/4+\epsilon},
\end{equation}
which holds for $X, T>2$.
Applying~\eqref{eq:tailsum} and~\eqref{eq:stxbound} in~\eqref{eq:exkdef}
then yields
\begin{equation}\label{eq:exk}
    E(X, k)\ll X^{9/4+\epsilon}Y^{-1}+X^{1+\epsilon}.
\end{equation}
Next, we use~\eqref{eq:psixkasymp},~\eqref{eq:exk} and
Theorem~\ref{intro:theorem1} in~\eqref{eq:smooth} to bound
\[
    E_{\Gamma}(X) \ll X^{9/4+\epsilon}Y^{-1}+X^{1-\beta+\epsilon}Y+X^{1+\epsilon}.
\]
Balancing the first two terms with $Y=X^{5/8+\beta/2}$ gives
\[
    E_{\Gamma}(X)\ll X^{13/8-\beta/2+\epsilon}+X^{1+\epsilon}.
\]
Recalling that $Y=X^{\nu}$, we have
\begin{equation*}
    \nu = \frac{5+4\beta}{8},\qquad  \beta = \frac{8\nu-5}{4}.
\end{equation*}
We also have from~\eqref{eq:sigmaeq} that
\[\beta=\nu-\frac{20(1-\sigma)}{3-\sigma}=\frac{(1-\sigma)(3\nu-1)}{4-3\sigma}.\]
Solving this system gives
\[\sigma=\frac{1}{472}\bigl(619-\sqrt{\num{31 049}}\bigr)\approx 0.93812,\]
\[\nu = \frac{1}{32}\bigl(197 - \sqrt{\num{31 049}}\,\bigr)\approx 0.649773,\]
and therefore
\[\beta/2 = \frac{1}{32}\bigl(177-\sqrt{\num{31 049}}\,\bigr)\approx 0.024773,\]
which concludes the proof of Corollary~\ref{cor:uncond}.

Corollaries \ref{cor:11/7} and \ref{cor:34/23} are proved with an identical argument.
First, we
use~\eqref{eq:shortintervalsub} instead of \eqref{intro:thm:eq}.
We obtain
\[
    E_{\Gamma}(X)
    \ll
    X^{9/4+\epsilon}Y^{-1} + X^{(4\theta+6)/5+\epsilon}Y^{2/5} + X^{1+\epsilon}.
\]
We balance this by choosing $Y=X^{(21-16\theta)/28}$ and get
\[
    E_{\Gamma}(X)\ll X^{3/2+4\theta/7+\epsilon},
\]
which proves Corollary~\ref{cor:11/7}.

Finally, for Corollary \ref{cor:34/23}, we recall that the assumption
\eqref{intro:meanlindelof} implies the estimate
(see \cite[p.~792]{koyama_prime_2006} and \cite[p.~5363]{balkanova_prime_2018-1})
\[
    S(T,X) \ll T^{7/4+\eps}X^{1/4+\eps} + T^{2}.
\]
Using this instead of \eqref{eq:stxbound} leads to
\[
    E_\Gamma(X) \ll X^{3/2+(24\theta-1)/46+\epsilon},
\]
which concludes the proof.

\begin{remark}
    The trivial bound in~\eqref{eq:stx} is $S(T, X)\ll T^{3}$.
    If we use this in the argument above, we obtain
    $E_{\Gamma}(X)\ll X^{3/2+2\theta/3+\epsilon}$.
    Notice that with the convexity bound $\theta=1/4$
    this recovers Sarnak's exponent 5/3~\eqref{eq:sarnak}, while
    the Burgess bound $\theta=3/16$ would yield the exponent 13/8
    as in~\cite{balkanova_prime_2018-1}
    (with a different proof, cf.~\cite[(17)]{soundararajan_prime_2013}).
\end{remark}


\end{document}